\documentclass{amsart}
\usepackage{amssymb}

\usepackage{amscd}

\input{tcilatex}
\input tcilatex

\begin{document}
\author{Marek Kara\'{s}, jakub Zygad\l o}
\title{On multidegree of tame and wild automorphisms of $\Bbb{C}^{3}$}
\keywords{polynomial automorphism, tame automorphism, wild automorphism, multidegree.\\
\textit{2000 Mathematics Subject Classification:} 14Rxx,14R10}
\date{}
\maketitle

\begin{abstract}
In this note we show that the set $\limfunc{mdeg}(\limfunc{Aut}(\Bbb{C}%
^{3}))\backslash \limfunc{mdeg}(\limfunc{Tame}(\Bbb{C}^{3}))$ is not empty.
Moreover we show that this set has infinitely many elements. Since for the
famous Nagata's example $N$ of wild automorphism, $\limfunc{mdeg}%
N=(5,3,1)\in $ $\limfunc{mdeg}(\limfunc{Tame}(\Bbb{C}^{3})),$ and since for
other known examples of wild automorphisms the multidegree is of the form $%
(1,d_{2},d_{3})$ (after permutation if neccesary), then we give the very
first exmple of wild automorphism $F\,$of $\Bbb{C}^{3}$ with $\limfunc{mdeg}%
F\notin $ $\limfunc{mdeg}(\limfunc{Tame}(\Bbb{C}^{3})).$

We also show that, if $d_{1},d_{2}$ are odd numbers such that $\gcd \left(
d_{1},d_{2}\right) =1,$ then $\left( d_{1},d_{2},d_{3}\right) \in \limfunc{%
mdeg}(\limfunc{Tame}(\Bbb{C}^{3}))$ if and only if $d_{3}\in d_{1}\Bbb{N}%
+d_{2}\Bbb{N}.$ This a crucial fact that we use in the proof of the main
result.
\end{abstract}

\section{Introduction}

Let us recall that a tame automorphism is, by definition, a composition of
linear automorphisms and triangular automorphisms, where a triangular
automorphism is a mapping of the following form 
\begin{equation*}
T:\Bbb{C}^{n}\ni \left\{ 
\begin{array}{l}
x_{1} \\ 
x_{2} \\ 
\vdots \\ 
x_{n}
\end{array}
\right\} \mapsto \left\{ 
\begin{array}{l}
x_{1} \\ 
x_{2}+f_{2}(x_{1}) \\ 
\vdots \\ 
x_{n}+f_{n}(x_{1},\ldots ,x_{n-1})
\end{array}
\right\} \in \Bbb{C}^{n}.
\end{equation*}
Recall, also, that an automorphism is called wild if it is not tame.

By multidegree of any polynomial mapping $F=(F_{1},\ldots ,F_{n}):\Bbb{C}%
^{n}\rightarrow \Bbb{C}^{n}$, denoted $\limfunc{mdeg}F,$ we call the
sequence $(\deg F_{1},\ldots ,\deg F_{n}).$ By $\limfunc{Tame}(\Bbb{C}^{n})$
we will denote the group of all tame automorphimsm of $\Bbb{C}^{n},$ and by $%
\limfunc{mdeg}$ the mapping from the set of all polynomial endomorphisms of $%
\Bbb{C}^{n}$ into the set $\Bbb{N}^{n}.$ In \cite{Karas1} it was proven that 
$(3,4,5),(3,5,7),(4,5,7),(4,5,11)\notin \limfunc{mdeg}(\limfunc{Tame}(\Bbb{C}%
^{3})).$ Next in \cite{Karas3} it was proved that $\left(
3,d_{2},d_{3}\right) \in \limfunc{mdeg}(\limfunc{Tame}(\Bbb{C}^{3})),$ for $%
3\leq d_{2}\leq d_{3},$ if and only if $3|d_{2}$ or $d_{3}\in 3\Bbb{N}+d_{2}%
\Bbb{N},$ and in \cite{Karas2} it was shoven that for $d_{3}\geq
d_{2}>d_{1}\geq 3,$ where $d_{1}$ and $d_{2}$ are prime numbers, $%
(d_{1},d_{2},d_{3})\in \limfunc{mdeg}(\limfunc{Tame}(\Bbb{C}^{3}))$ if and
only if $d_{3}\in d_{1}\Bbb{N}+d_{2}\Bbb{N}.$ In this paper we give
generalization of this result (Theorem \ref{tw_d1_d2_odd} below), and using
this fact we show the following theorem.

\begin{theorem}
\label{tw_main}\label{main}The set $\limfunc{mdeg}(\limfunc{Aut}(\Bbb{C}%
^{3}))\backslash \limfunc{mdeg}(\limfunc{Tame}(\Bbb{C}^{3}))$ has infinitely
many elements.
\end{theorem}

Notice that the existence of wild automorphism does not imply the above
result. For example the famous Nagata's exmple is wild, but its muldidegree
is (after pemutetion) $(1,3,5)\in \limfunc{mdeg}(\limfunc{Tame}(\Bbb{C}%
^{3})),$ because for instance the mapp $\Bbb{C}^{3}\ni (x,y,z)\mapsto
(x,y+x^{3},z+x^{5})\in \Bbb{C}^{3}$ is a tame automorphism.

\section{{\textsl{Tame automorphisms of }}$\Bbb{C}^{3}${\textsl{with
multidegree of the form }}$(d_{1},d_{2},d_{3})$ {\textsl{with }}$\gcd
(d_{1},d_{2})=1${\textsl{\ and odd }}$d_{1},d_{2},$}

In the proof of Theorem \ref{main} we will use the following genrealization
of the result of \cite{Karas2}.

\begin{theorem}
\label{tw_odd_odd_gcd_1}\label{tw_d1_d2_odd}Let $d_{3}\geq d_{2}>d_{1}\geq 3$
be positive integers. If $d_{1}$ and $d_{2}$ are odd numbers such that $\gcd
\left( d_{1},d_{2}\right) =1$, then $(d_{1},d_{2},d_{3})\in \limfunc{mdeg}(%
\limfunc{Tame}(\Bbb{C}^{3}))$ if and only if $d_{3}\in d_{1}\Bbb{N}+d_{2}%
\Bbb{N},$ i.e. if and only if $d_{3}$ is a linear combination of $d_{1}$ and 
$d_{2}$ with coefficients in $\Bbb{N}.$
\end{theorem}

In the proof of this theorem we will need the following results that we
include here for the convenience of the reader.

\begin{theorem}
\label{tw_sywester}If $d_{1},d_{2}$ are positive integers such that $\gcd
(d_{1},d_{2})=1,$ then for every integer $k\geq (d_{1}-1)(d_{2}-1)$ there
are $k_{1},k_{2}\in \Bbb{N}$ such that 
\begin{equation*}
k=k_{1}d_{1}+k_{2}d_{2}.
\end{equation*}
Moreover $(d_{1}-1)(d_{2}-1)-1\notin d_{1}\Bbb{N}+d_{2}\Bbb{N}.$
\end{theorem}

\begin{proposition}
\textit{(\cite{Karas1}, Proposition 2.2) }\label{prop_sum_d_i}If for a
sequence of integers $1\leq d_{1}\leq \ldots \leq d_{n}$ there is $i\in
\{1,\ldots ,n\}$ such that 
\begin{equation*}
d_{i}=\sum_{j=1}^{i-1}k_{j}d_{j}\qquad \text{with }k_{j}\in \Bbb{N},
\end{equation*}
then there exists a tame automorphism $F$ of $\Bbb{C}^{n}$ with $\limfunc{%
mdeg}F=(d_{1},\ldots ,d_{n}).$
\end{proposition}

\begin{proposition}
\textit{(\cite{Karas3}, Proposition 2.4) }\label{prop_deg_g_fg}Let $f,g\in 
\Bbb{C}[X_{1},\ldots ,X_{n}]$ be such that $f,g$ are algebraically
independent and $\overline{f}\notin \Bbb{C}\left[ \overline{g}\right] ,%
\overline{g}\notin \Bbb{C}\left[ \overline{f}\right] $ ($\overline{h}$ means
the highest homogeneous part of $h$). Assume that $\deg f<\deg g,$ put 
\begin{equation*}
p=\frac{\deg f}{\gcd \left( \deg f,\deg g\right) },
\end{equation*}
and let $G(x,y)\in \Bbb{C}[x,y]$ with $\deg _{y}G(x,y)=pq+r,0\leq r<p.$ Then 
\begin{equation*}
\deg G(f,g)\geq q\left( p\deg g-\deg g-\deg f+\deg [f,g]\right) +r\deg g.
\end{equation*}
\end{proposition}

In the above proposition $[f,g]$ means the Poisson bracket of $f$ and $g$
defined as the following formal sum: 
\begin{equation*}
\sum_{1\leq i<j\leq n}\left( \frac{\partial f}{\partial x_{i}}\frac{\partial
g}{\partial x_{j}}-\frac{\partial f}{\partial x_{j}}\frac{\partial g}{%
\partial x_{i}}\right) \left[ X_{i},X_{j}\right]
\end{equation*}
and: 
\begin{equation*}
\deg \left[ f,g\right] =\underset{1\leq i<j\leq n}{\max }\left\{ \left( 
\frac{\partial f}{\partial x_{i}}\frac{\partial g}{\partial x_{j}}-\frac{%
\partial f}{\partial x_{j}}\frac{\partial g}{\partial x_{i}}\right) \left[
X_{i},X_{j}\right] \right\} ,
\end{equation*}
where by definition $\deg \left[ X_{i},X_{j}\right] =2$ for $i\neq j$ and $%
\deg 0=-\infty .$

From the definition of the Poisson bracket we have 
\begin{equation*}
\deg \left[ f,g\right] \leq \deg f+\deg g
\end{equation*}
and by Proposition 1.2.9 of \cite{van den Essen}: 
\begin{equation*}
\deg [f,g]=2+\underset{1\leq i<j\leq n}{\max }\deg \left( \frac{\partial f}{%
\partial x_{i}}\frac{\partial g}{\partial x_{j}}-\frac{\partial f}{\partial
x_{j}}\frac{\partial g}{\partial x_{i}}\right)
\end{equation*}
if $f,g$ are algebraically independent, and $\deg [f,g]=0$ if $f,g$ are
algebraically dependent.

The last result we will need is the following theorem.

\begin{theorem}
\label{tw_type_1-4}\textit{(\cite{sh umb1}, Theorem 3) }Let $%
F=(F_{1},F_{2},F_{3})\,$be a tame automorphism of $\Bbb{C}^{3}.$ If $\deg
F_{1}+\deg F_{2}+\deg F_{3}>3$ (in other words if $F$ is not a linear
automorphism), then $F$ admits either an elementary reduction or a reduction
of types I-IV (see \cite{sh umb1} Definitions 2-4).
\end{theorem}

Let us, also, recall that an automorphism $F=(F_{1},F_{2},F_{3})$ admits an
elementary reduction if there exists a polynomial $g\in \Bbb{C}[x,y]$ and a
permutation $\sigma $ of the set $\{1,2,3\}$ such that $\deg (F_{\sigma
(1)}-g(F_{\sigma (2)},F_{\sigma (3)}))<\deg F_{\sigma (1)}.$ In other words
if there exists an elementary automorphism $\tau :\Bbb{C}^{3}\rightarrow 
\Bbb{C}^{3}$ such that $\limfunc{mdeg}\left( \tau \circ F\right) <\limfunc{%
mdeg}F,$ where $\left( d_{1},\ldots ,d_{n}\right) <\left( k_{1},\ldots
,k_{n}\right) $ means that $d_{l}\leq k_{l}$ for all $l\in \left\{ 1,\ldots
,n\right\} $ and $d_{i}<k_{i}$ for at least one $i\in \left\{ 1,\ldots
,n\right\} .$ Recall also that a mapping $\tau =\left( \tau _{1},\ldots
,\tau _{n}\right) :\Bbb{C}^{n}\rightarrow \Bbb{C}^{n}$ is called an
elementary automorphism if there exists $i\in \left\{ 1,\ldots ,n\right\} $
such that 
\begin{equation*}
\tau _{j}\left( x_{1},\ldots ,x_{n}\right) =\left\{ 
\begin{array}{ll}
x_{j} & \text{for }j\neq i \\ 
x_{i}+g\left( x_{1},\ldots ,x_{i-1},x_{i+1},\ldots ,x_{n}\right) & \text{for 
}j=i
\end{array}
.\right.
\end{equation*}

\begin{proof}
\textit{(of Theorem \ref{tw_d1_d2_odd}) }Assume that $F=(F_{1},F_{2},F_{3})$
is an automorphism of $\Bbb{C}^{3}$ such that $\limfunc{mdeg}%
F=(d_{1},d_{2},d_{3}).$ Assume, also, that $d_{3}\notin d_{1}\Bbb{N}+d_{2}%
\Bbb{N}.$ By Theorem \ref{tw_sywester} we have: 
\begin{equation}
d_{3}<(d_{1}-1)(d_{2}-1).  \label{d_3 male_a}
\end{equation}
First of all we show that this hypothetical automorphism $F$ does not admit
reductions of type I-IV.

By the definitions of reductions of types I-IV (see \cite{sh umb1}
Definitions 2-4), if $F=(F_{1},F_{2},F_{3})$ admits a reduction of these
types, then $2|\deg F_{i}$ for some $i\in \{1,2,3\}.$ Thus if $d_{3}$ is
odd, then $F$ does not admit a reduction of types I-IV. Assume that $%
d_{3}=2n $ for some positive integers $n.$

If we assume that $F$ admits a reduction of type I or II, then by the
definition (see \cite{sh umb1} Definition 2 and 3) we have $d_{1}=sn$ or $%
d_{2}=sn$ for some odd $s\geq 3.$ Since $d_{1},d_{2}\leq d_{3}=2n<sn,$ then
we obtain a contradiction.

And, if we assume that $F$ admits a reduction of type III or IV, then by the
definition (see \cite{sh umb1} Definition 4) we have: 
\begin{equation*}
n<d_{1}\leq \tfrac{3}{2}n,\qquad d_{2}=3n
\end{equation*}
or 
\begin{equation*}
d_{1}=\tfrac{3}{2}n,\qquad \tfrac{5}{2}n<d_{2}\leq 3n.
\end{equation*}
Since $d_{1},d_{2}\leq d_{3}=2n<\tfrac{5}{2}n,3n,$ then we obtain a
contradiction. Thus we have proved that our hypothetical automorphism $F$
does not admit a reduction of types I-IV.

Now we will show that it, also, does not admit an elementary reduction.

Assume, by a contrary, that 
\begin{equation*}
(F_{1},F_{2},F_{3}-g(F_{1},F_{2})),
\end{equation*}
where $g\in \Bbb{C}[x,y],$ is an elementary reduction of $%
(F_{1},F_{2},F_{3}).$ Then we have $\deg g(F_{1},F_{2})=\deg F_{3}=d_{3}.$
But, by Proposition \ref{prop_deg_g_fg}, we have 
\begin{equation*}
\deg g(F_{1},F_{2})\geq q(d_{1}d_{2}-d_{1}-d_{2}+\deg [F_{1},F_{2}])+rd_{2},
\end{equation*}
where $\deg _{y}g(x,y)=qd_{1}+r$ with $0\leq r<d_{1}.\,$ Since $F_{1},F_{2}$
are algebraically independent, $\deg [F_{1},F_{2}]\geq 2$ and then 
\begin{equation*}
d_{1}d_{2}-d_{1}-d_{2}+\deg [F_{1},F_{2}]\geq
d_{1}d_{2}-d_{1}-d_{2}+2>(d_{1}-1)(d_{2}-1).
\end{equation*}
This and (\ref{d_3 male_a}) follows that $q=0,$ and that: 
\begin{equation*}
g(x,y)=\sum_{i=0}^{d_{1}-1}g_{i}(x)y^{i}.
\end{equation*}
Since $\func{lcm}(d_{1},d_{2})=d_{1}d_{2},$ then the sets 
\begin{equation*}
d_{1}\Bbb{N},d_{2}+d_{1}\Bbb{N},\ldots ,(d_{1}-1)d_{2}+d_{1}\Bbb{N}
\end{equation*}
are pair wise disjoint. This follows that: 
\begin{equation*}
d_{3}=\deg \left( \sum_{i=0}^{d_{1}-1}g_{i}(F_{1})F_{2}^{i}\right) =%
\underset{i=0,\ldots ,d_{1}-1}{\max }\left( \deg F_{1}\deg g_{i}+i\deg
F_{2}\right) .
\end{equation*}
Thus 
\begin{equation*}
d_{3}\in \bigcup_{r=0}^{d_{1}-1}\left( rd_{2}+d_{1}\Bbb{N}\right) \subset
d_{1}\Bbb{N}+d_{2}\Bbb{N}.
\end{equation*}
This is a contradiction with $d_{3}\notin d_{1}\Bbb{N}+d_{2}\Bbb{N}.$

Now, assume that 
\begin{equation*}
(F_{1},F_{2}-g(F_{1},F_{3}),F_{3}),
\end{equation*}
where $g\in \Bbb{C}[x,y],$ is an elementary reduction of $%
(F_{1},F_{2},F_{3}).$ Since $d_{3}\notin d_{1}\Bbb{N}+d_{2}\Bbb{N},$ then $%
d_{1}\nmid d_{3}.$ This follows that 
\begin{equation*}
p=\frac{d_{1}}{\gcd \left( d_{1},d_{3}\right) }>1.
\end{equation*}
Since $d_{1}\,$is odd number, we also have $p\neq 2.$ Thus by Proposition 
\ref{prop_deg_g_fg} we have 
\begin{equation*}
\deg g(F_{1},F_{3})\geq q(pd_{3}-d_{3}-d_{1}+\deg [F_{1},F_{3}])+rd_{3},
\end{equation*}
where $\deg _{y}g(x,y)=qp+r$ with $0\leq r<p.$ Since $p\geq 3,$ then $%
pd_{3}-d_{3}-d_{1}+\deg [F_{1},F_{3}]\geq 2d_{3}-d_{1}+2>d_{3}.$ Since we
want to have $\deg g(F_{1},F_{3})=d_{2},$ then $q=r=0,$ and then $%
g(x,y)=g(x).$ This means that $d_{2}=\deg g\left( F_{1},F_{3}\right) =\deg
g\left( F_{1}\right) .$ But this is a contradiction with $d_{2}\notin d_{1}%
\Bbb{N}$ (remember that $\gcd \left( d_{1},d_{2}\right) =1$).

Finally, if we assume that $(F_{1}-g(F_{2},F_{3}),F_{2},F_{3})$ is an
elementary reduction of $(F_{1},F_{2},F_{3}),$ then in the same way as in
the previous case we obtain a contradiction.
\end{proof}

\section{Proof of the theorem}

Let $N:\Bbb{C}^{3}\ni (x,y,z)\mapsto
(x+2y(y^{2}+zx)-z(y^{2}+zx)^{2},y-z(y^{2}+zx),z)\in \Bbb{C}^{3}$ be the
Nagata's exmple and let $T:\Bbb{C}^{3}\ni (x,y,z)\mapsto (z,y,x)\in \Bbb{C}%
^{3}.$ We start with the following lemma.

\begin{lemma}
For all $n\in \Bbb{N}$ we have $\limfunc{mdeg}(T\circ
N)^{n}=(4n-3,4n-1,4n+1).$
\end{lemma}

\begin{proof}
We have $T\circ N(x,y,z)=(z,y-z(y^{2}+zx),x+2y(y^{2}+zx)-z(y^{2}+zx)^{2}),$
so the above equality is true for $n=1.$ Let $(f_{n},g_{n},h_{n})=T\circ N$
for $f_{n},g_{n},h_{n}\in \Bbb{C[}X,Y,Z].$ One can see that $%
g_{1}^{2}+h_{1}f_{1}=Y^{2}+ZX,$ and by induction that $%
g_{n}^{2}+h_{n}f_{n}=Y^{2}+ZX$ for any $n\in \Bbb{N}\backslash \{0\}.\Bbb{\,}
$Thus 
\begin{eqnarray*}
(f_{n+1},g_{n+1},h_{n+1}) = \left( T\circ N\right) \left(
f_{n},g_{n},h_{n}\right) \\
=\left( h_{n},g_{n}-h_{n}\left( g_{n}^{2}+h_{n}f_{n}\right)
,f_{n}+2h_{n}\left( g_{n}^{2}+h_{n}f_{n}\right) -h_{n}\left(
g_{n}^{2}+h_{n}f_{n}\right) ^{2}\right) \\
=\left( h_{n},g_{n}-h_{n}\left( Y^{2}+ZX\right) ,f_{n}+2h_{n}\left(
Y^{2}+ZX\right) -h_{n}\left( Y^{2}+ZX\right) ^{2}\right) .
\end{eqnarray*}
So if we assume that $\limfunc{mdeg}(f_{n},g_{n},h_{n})=(4n-3,4n-1,4n+1),$
we obtain \allowbreak $\limfunc{mdeg}(f_{n+1},g_{n+1},h_{n+1})=(4n+1,\left(
4n+1\right) +2,\left( 4n+1\right) +2\cdot 2)=(4(n+1)-3,4(n+1)-1,4(n+1)+1).$
\end{proof}

By the above lemma and Theorem \ref{tw_d1_d2_odd} we obtain the following
theorem.

\begin{theorem}
For every $n\in \Bbb{N}$ the automorphism $(T\circ N)^{n}$ is wild.
\end{theorem}

\begin{proof}
For $n=1$ this is the result of Shestakov and Umirbayev \cite{sh umb1,sh
umb2}. So we can assume that $n\geq 2.$ The numbers $4n-3,4n-1$ are odd and $%
\gcd (4n-3,4n-1)=\gcd (4n-3,2)=1.$ Since $4n-3>2,$ $4n+1\notin (4n-3)\Bbb{N}%
+(4n-1)\Bbb{N}.$ Then by Theorem \ref{tw_d1_d2_odd} $(4n-3,4n-1,4n+1)\notin 
\limfunc{mdeg}(\limfunc{Tame}(\Bbb{C}^{3}))$ for $n>1.$ This proves Theorem 
\ref{main} and that $(T\circ N)^{n}$ is not a tame automphism.
\end{proof}

Let us notice that in the proof of the above theorem we, aslo, proved that
\begin{equation*}
\left\{ \left( 4n-3,4n-1,4n+1\right) :n\in \Bbb{N},n\geq 2\right\} \subset 
\limfunc{mdeg}(\limfunc{Aut}(\Bbb{C}^{3}))\backslash \limfunc{mdeg}(\limfunc{%
Tame}(\Bbb{C}^{3})).
\end{equation*}
This proves Theorem \ref{tw_main}.

\vspace{1cm}

\textsc{Marek Kara\'{s}\newline
Instytut Matematyki\newline
Uniwersytetu Jagiello\'{n}skiego\newline
ul. \L ojasiewicza 6}\newline
\textsc{30-348 Krak\'{o}w\newline
Poland\newline
} e-mail: Marek.Karas@im.uj.edu.pl\vspace{0.5cm}

and\vspace{0.5cm}

\textsc{Jakub Zygad\l o\newline
Instytut Informatyki\newline
Uniwersytetu Jagiello\'{n}skiego\newline
ul. \L ojasiewicza 6}\newline
\textsc{30-348 Krak\'{o}w\newline
Poland\newline
} e-mail: Jakub.Zygadlo@ii.uj.edu.p

\end{document}